\documentclass[11pt, reqno]{amsart}

\author[P.~Leonetti]{Paolo Leonetti}
\address{Universit\`a ``Luigi Bocconi''\\Department of Statistics\\Milan, Italy}
\email{leonetti.paolo@gmail.com}
\urladdr{\url{https://sites.google.com/site/leonettipaolo/}}


\title{Thinnable Ideals and Invariance of Cluster Points}

\usepackage{amsmath}
\usepackage{amssymb}
\usepackage{amsthm}
\usepackage[left=3.5cm, right=3.5cm, paperheight=11.8in]{geometry}
\usepackage{hyperref}
\usepackage{fancyhdr}
\usepackage{enumitem}
\usepackage{comment}
\usepackage{nicefrac}
\usepackage{mathrsfs}
\usepackage{graphicx}
\usepackage[utf8]{inputenc}

\AtBeginDocument{%
   \def\MR#1{}
}

\newtheorem{thm}{Theorem}[section]
\newtheorem{cor}[thm]{Corollary}

\newtheorem{prop}[thm]{Proposition}

\theoremstyle{definition} 
\newtheorem{defi}[thm]{Definition}
\let\olddefi\defi
\renewcommand{\defi}{\olddefi\normalfont}
\newtheorem{example}[thm]{Example}
\let\oldexample\example
\renewcommand{\example}{\oldexample\normalfont}
\newtheorem{rmk}[thm]{Remark}
\let\oldrmk\rmk
\renewcommand{\rmk}{\oldrmk\normalfont}

\hypersetup{
    pdftitle={Thinnable Ideals and Invariance of Cluster Points},
    pdfauthor={Paolo Leonetti},
    pdfmenubar=false,
    pdffitwindow=true,
    pdfstartview=FitH,
    colorlinks=true,
    linkcolor=blue,
    citecolor=green,
    urlcolor=cyan
}

\providecommand{\MR}[1]{}

\providecommand{\MR}{\relax\ifhmode\unskip\space\fi MR }

\providecommand{\href}[2]{#2}
                                 
\begin{document}

\maketitle
\thispagestyle{empty}

\begin{abstract}
\noindent 
We define a class of so-called thinnable ideals $\mathcal{I}$ on the positive integers which includes several well-known examples, e.g., the collection of sets with zero asymptotic density, sets with zero logarithmic density, and several summable ideals. 
Given a sequence $(x_n)$ taking values in a separable metric space and a thinnable ideal $\mathcal{I}$, it is shown that the set of $\mathcal{I}$-cluster points of $(x_n)$ is equal to the set of $\mathcal{I}$-cluster points of almost all its subsequences, in the sense of Lebesgue measure.

Lastly, we obtain a characterization of ideal convergence, which improves the main result in [Trans. Amer. Math. Soc. \textbf{347} (1995), 1811--1819].
\end{abstract}


\section{Introduction}\label{sec:intro}

It is well known that the set of ordinary limit points of ``almost every'' subsequence of a real sequence $(x_n)$ coincides with the set of ordinary limit points of the original sequence, in the sense of Lebesgue measure, see Buck \cite{MR0009997}. In the same direction, we prove its analogue for ideal cluster points.

To this aim, let $\mathcal{I}$ be an ideal on the positive integers $\mathbf{N}$, that is, a family of subsets of $\mathbf{N}$ closed under taking finite unions and subsets of its elements. It is assumed that $\mathcal{I}$ contains the collection $\mathrm{Fin}$ of finite subsets of $\mathbf{N}$ and it is different from the whole power set of $\mathbf{N}$. Note that the collection of subsets with zero asymptotic density, i.e., 
$$
\mathcal{I}_{0}:=\left\{S\subseteq \mathbf{N}: \lim_{n\to \infty}\frac{|S\cap [1,n]|}{n}=0\right\},
$$
is an ideal. 
Let also $x=(x_n)$ be a sequence taking values in a topological space $X$. We denote by $\Gamma_x(\mathcal{I})$ the set of $\mathcal{I}$-cluster points of $x$, that is, the set of all $\ell \in X$ such that 
$$
\{n: x_n \in U\} \notin \mathcal{I}
$$
for all neighborhoods $U$ of $\ell$. Statistical cluster points (that is, $\mathcal{I}_{0}$-cluster points) of real sequences were introduced by Fridy \cite{MR1181163}, cf. also \cite{MR2463821, MR1416085, LMxyz}. 
However, it is worth noting that ideal cluster points have been studied much before under a different name. Indeed, as it follows by \cite[Theorem 4.2]{LMxyz}, they correspond to classical ``cluster points'' of a filter $\mathscr{F}$ on $\mathbf{R}$ (depending on $x$), cf. \cite[Definition 2, p.69]{MR1726779}.



As anticipated, the main question addressed here is 
to find suitable conditions on $X$ and $\mathcal{I}$ such that  
the set of $\mathcal{I}$-cluster points of a sequence $(x_n)$ is equal to the set of $\mathcal{I}$-cluster points of ``almost all'' its subsequences. Finally, we obtain a characterization of ideal convergence. Related results were obtained in \cite{MR0010616, MR0316930, MR1260176, 
Miller18, MR3469337, 
MR1924673}.

\section{Thinnability}\label{sec:thinnable}

Given $k \in \mathbf{N}$ and \emph{infinite} sets $A,B\subseteq \mathbf{N}$ with canonical enumeration $\{a_n: n \in \mathbf{N}\}$ and $\{b_n: n \in \mathbf{N}\}$, respectively, we write $A\le B$ if $a_n\le b_n$ for all $n \in \mathbf{N}$ and define
$$
A_B:=\{a_b: b \in B\} \,\,\text{ and }\,\,kA:=\{ka: a \in A\}.
$$

\begin{defi}\label{def:thinn}
An ideal $\mathcal{I}$ is said to be \emph{weakly thinnable} if $A_B \notin \mathcal{I}$ whenever $A\subseteq \mathbf{N}$ admits non-zero asymptotic density and $B\notin \mathcal{I}$.

If, in addition, also $B_A\notin \mathcal{I}$ and $X\notin \mathcal{I}$ whenever $X \le Y$ and $Y\notin \mathcal{I}$, then $\mathcal{I}$ is said to be \emph{thinnable}.
%
%
%
\end{defi}
\begin{defi}\label{def:thinn}
An ideal $\mathcal{I}$ is said to be \emph{strechable} if $kA \notin \mathcal{I}$ for all $k \in \mathbf{N}$ and $A\notin \mathcal{I}$.
\end{defi}
The terminology has been suggested from the related properties of finitely additive measures on $\mathbf{N}$ studied in \cite{MR1192311}. 
In this regard, $\mathrm{Fin}$ is thinnable and strechable. 

This is the case of several other ideals: 
\begin{prop}\label{thm:thinnable}
Let $f: \mathbf{N} \to (0,\infty)$ be a definitively non-increasing function such that $\sum_{n\ge 1}f(n)=\infty$. Define the summable ideal
$$
\mathcal{I}_f:=\left\{S\subseteq \mathbf{N}: \sum_{n \in S}f(n)<\infty\right\}.
$$
Then $\mathcal{I}_f$ is thinnable provided $\mathcal{I}_f$ is strechable.
%

In addition, suppose that 
\begin{equation}\label{eq:addlimit}
\liminf_{n\to \infty}\frac{\sum_{i \in [1,n]}f(i)}{\sum_{i \in [1,kn]}f(i)}\neq 0\,\,\,\text{ for all }k \in \mathbf{N}
\end{equation} 
and define the Erd{\H o}s--Ulam ideal
$$
\mathscr{E}_f:=\left\{S\subseteq \mathbf{N}: \lim_{n\to \infty}\frac{\sum_{i \in S\cap [1,n]}f(i)}{\sum_{i \in [1,n]}f(i)}=0\right\}.
$$
Then $\mathscr{E}_f$ is thinnable provided $\mathscr{E}_f$ is strechable.
\end{prop}
\begin{proof}
Let us suppose that $A=\{a_n: n\in \mathbf{N}\}$ admits asymptotic density $c>0$ and $B=\{b_n: n\in \mathbf{N}\}\notin \mathcal{I}_f$, that is, $\sum_{n\ge 1}f(b_n)=\infty$. Define the integer $k:=\lfloor 1/c\rfloor+1 \ge 2$ and note that $\sum_{n\ge 1}f(kb_n)=\infty$ by the fact that $\mathcal{I}_f$ is strechable. Then $a_n=\frac{1}{c}n(1+o(1))$ as $n\to \infty$, which implies
\begin{equation}\label{eq:eq1}
\textstyle \sum_{n\ge 1}f(a_{b_n})\ge O(1)+\sum_{n\ge 1}f(kb_{n})=\infty,
\end{equation}
i.e., $A_B \notin \mathcal{I}_f$, hence $\mathcal{I}_f$ is weakly thinnable. 
Moreover, observe that 
\begin{equation}\label{eq:ineqprop}
\sum_{n\equiv 1\bmod{k}}f(b_n) \ge \sum_{n\equiv 2\bmod{k}}f(b_n)\ge \cdots \ge \sum_{n\equiv 0\bmod{k}}f(b_n) \ge \sum_{\substack{n\equiv 1\bmod{k}, \\  n\neq 1}}f(b_n)
\end{equation}
and note that the first sum is finite if and only if the last sum is finite. Since $I\notin \mathcal{I}_f$, then all the above sums are infinite, which implies that 
$$
\sum_{n\ge 1}f(b_{a_n})\ge O(1)+\sum_{n\ge 1}f(b_{kn})=\infty,
$$
i.e., $B_A \notin \mathcal{I}_f$. Lastly, given infinite sets $X,Y \subseteq \mathbf{N}$ with $X\le Y$ and $X \in \mathcal{I}_f$, we have $\sum_{y \in Y}f(y)\le \sum_{x \in X}f(x)<\infty$. Therefore $\mathcal{I}_f$ is thinnable. 

The proof of the second part is similar, where \eqref{eq:eq1} is replaced by
$$
\textstyle \sum_{a_{b_n} \le x}f(a_{b_n}) \ge O(1)+\sum_{{b_n} \le x/k}f(k{b_n}).
$$
Moreover, $B\notin \mathscr{E}_f$ implies $kB \notin \mathscr{E}_f$ by the hypothesis of strechability, i.e.,
$$
\textstyle \sum_{{b_n} \le x/k}f(k{b_n}) \neq o\left(\sum_{i \le x/k}f(i)\right);
$$ 
thanks to \eqref{eq:addlimit}, we conclude that
$$
\textstyle \sum_{{b_n} \le x/k}f(k{b_n}) \neq o\left(\sum_{i \le x}f(i)\right),
$$
hence $A_B \notin \mathscr{E}_f$, which shows that $\mathscr{E}_f$ is weakly thinnable. In addition, we get
\begin{displaymath}
\frac{f(b_{a_1})+\cdots+f(b_{a_n})}{f(1)+\cdots+f(b_{a_n})}\ge \frac{O(1)+f(b_{k})+\cdots+f(b_{kn})}{f(1)+\cdots+f(b_{kn})} \not\to 0,
\end{displaymath}
so that $B_A \notin \mathscr{E}_f$, where the last $\not\to$ comes from a reasoning similar to \eqref{eq:ineqprop}. 
Finally, given infinite subsets $X,Y\subseteq \mathbf{N}$ with canonical enumeration $\{x_n:n \in \mathbf{N}\}$ and $\{y_n: n \in \mathbf{N}\}$, respectively, such that $X\le Y$ and $X \in \mathscr{E}_f$, it holds
$$
\frac{f(x_1)+\cdots+f(x_n)}{f(1)+\cdots+f(x_n)}\ge 
\frac{f(y_1)+\cdots+f(y_n)}{f(1)+\cdots+f(y_n)}  
$$
for all $n \in \mathbf{N}$ therefore $Y \in \mathscr{E}_f$.
\end{proof}

Given a real $\alpha \ge -1$, let $\mathcal{I}_\alpha$ be the collection of subsets with zero $\alpha$-density, that is,
\begin{equation}\label{eq:defalpha}
\mathcal{I}_\alpha:=\left\{S\subseteq \mathbf{N}: \mathrm{d}_\alpha^\star(S)=0\right\}, \,\,\text{ where }\,\,\mathrm{d}_\alpha^\star(S)=\limsup_{n\to \infty}\frac{\sum_{i \in S\cap [1,n]} i^\alpha}{\sum_{i \in [1,n]} i^\alpha}.
\end{equation}

\begin{prop}\label{prop:ialphaideals}
All ideals $\mathcal{I}_\alpha$ are thinnable.
\end{prop}
\begin{proof}
If $\alpha \in [-1,0]$, the claim follows by Proposition \ref{thm:thinnable} (we omit details). Hence, let us suppose hereafter than $\alpha>0$.  
Fix infinite sets $X,Y \subseteq \mathbf{N}$ with canonical enumerations $\{x_n: n \in \mathbf{N}\}$ and $\{y_n: n \in \mathbf{N}\}$, respectively, such that $Y \notin \mathcal{I}_\alpha$. Then, there exist an infinite set $S$ such that $|Y \cap [1,y_n]|\ge \lambda y_n$ for all $n \in S$, where $\lambda:=1-\left(1-\frac{1}{2}\mathrm{d}_\alpha^\star(Y)\right)^{\frac{1}{\alpha+1}}>0$. Indeed, in the opposite case, we would have that
$$
\frac{\alpha+1}{y_n^{\alpha+1}}\sum_{i\le n}y_i^\alpha \le \frac{\alpha+1}{y_n^{\alpha+1}}\sum_{i \in ((1-\lambda)y_n,y_n]}i^\alpha \le \left(1-(1-\lambda)^{\alpha+1}\right))(1+o(1))<\frac{2}{3}\mathrm{d}_\alpha^\star(Y)
$$
for all sufficiently large $n$. Since $|Y \cap [1,n]| \le |X \cap [1,n]|$ for all $n$, we conclude that
$$
\frac{1}{x_n^{\alpha+1}}\sum_{i\le n}x_i^\alpha \ge \frac{1}{x_n^{\alpha+1}}\sum_{i\le \lambda y_n}i^\alpha \ge \frac{1}{x_n^{\alpha+1}}\sum_{i\le \lambda x_n}i^\alpha \ge \frac{\lambda^{\alpha+1}}{2}
$$
for all large $n \in S$, so that $X \notin \mathcal{I}_\alpha$.

At this point, fix sets $A,B \subseteq \mathbf{N}$ 
with canonical enumerations $\{a_n: n \in \mathbf{N}\}$ and $\{b_n: n \in \mathbf{N}\}$, respectively, 
such that $A$ admits asymptotic density $c>0$ and $B \notin \mathcal{I}_\alpha$. 
Fix also $\varepsilon>0$ sufficiently small and note that there exists $n_0=n_0(\varepsilon) \in \mathbf{N}$ such that $(\nicefrac{1}{c}-\varepsilon)n \le a_n \le (\nicefrac{1}{c}+\varepsilon)n$ for all $n\ge n_0$. In particular, it follows that
$$
\frac{1}{a_{b_n}^{\alpha+1}}\sum_{k\le n}\left(a_{b_k}\right)^\alpha \ge \frac{1}{\left(\frac{1}{c}+\varepsilon\right)^{\alpha+1}{b_n}^{\alpha+1}}\left(O(1)+\sum_{n_0\le k\le n}\left(\frac{1}{c}-\varepsilon\right)^{\alpha}{b_k}^{\alpha}\right).
$$
Therefore, setting $\kappa:=\min\left\{\left(\frac{1}{c}+\varepsilon\right)^{-\alpha-1}, \left(\frac{1}{c}-\varepsilon\right)^{\alpha}\right\}>0$, we obtain
\begin{displaymath}
\begin{split}
\frac{\mathrm{d}_\alpha^\star(A_B)}{\alpha+1}&=\limsup_{n\to \infty}\frac{1}{a_{b_n}^{\alpha+1}}\sum_{k\le n}\left(a_{b_k}\right)^\alpha \ge \limsup_{n\to \infty}\frac{\kappa}{{b_n}^{\alpha+1}}\left(O(1)+\sum_{n_0\le k\le n}\kappa{b_k}^{\alpha}\right) \\
&= \kappa^2\limsup_{n\to \infty}\frac{1}{{b_n}^{\alpha+1}}\sum_{n_0\le k\le n}{b_k}^{\alpha}= \kappa^2 \,\frac{\mathrm{d}_\alpha^\star(B)}{\alpha+1}>0.
\end{split}
\end{displaymath}
This proves that $A_B \notin \mathcal{I}_\alpha$. Finally, let $k$ be an integer greater than $\nicefrac{1}{c}$ and note that $B_A \le B_{k\mathbf{N}}\setminus S$, for some finite set $S$. By the previous observation, it is sufficient to show that $B_{k\mathbf{N}} \notin \mathcal{I}_\alpha$ and this is straightforward by an analogous argument of \eqref{eq:ineqprop}.
\end{proof}

To mention another example, let $\mathcal{I}_{\mathfrak{p}}$ be the \emph{P\'olya ideal}, i.e., 
$$
\mathcal{I}_{\mathfrak{p}}:=\left\{S\subseteq \mathbf{N}: \mathfrak{p}^\star(S)=0\right\}, \,\,\text{ where }\,\,\mathfrak{p}^\star(S)=\lim_{s \to 1^-} \limsup_{n \to \infty} \frac{|S \cap [ns,n]|}{(1-s)n}.
$$
Among other things, the upper P\'olya density $\mathfrak p^\ast$ has found a number of remarkable applications in analysis and economic theory, see e.g. \cite{MR1545027}, \cite{MR0003208} and \cite{MR1656470}.

\begin{cor}\label{cor:polya}
The P\'olya ideal $\mathcal{I}_{\mathfrak{p}}$ is thinnable.
\end{cor}
\begin{proof}
The upper P\'olya density $\mathfrak p^\ast$ is the pointwise limit of the real net of the upper $\alpha$-densities $\mathrm{d}_\alpha^\star$ defined in \eqref{eq:defalpha}, see \cite[Theorem 4.3]{MR3278191}. 

Fix infinite sets $X,Y \subseteq \mathbf{N}$ with canonical enumerations $\{x_n: n \in \mathbf{N}\}$ and $\{y_n: n \in \mathbf{N}\}$, respectively, such that $Y \notin \mathcal{I}_\mathfrak{p}$. Then, there exists $\alpha>0$ such that $\mathrm{d}_\alpha^\star(Y)>0$ and, thanks to Proposition \ref{prop:ialphaideals}, we get $\mathrm{d}_\alpha^\star(X)>0$ as well. This implies that $X\notin \mathcal{I}_\mathfrak{p}$. Other properties can be shown similarly.
\end{proof}

Lastly, it is worth noting that there exist summable ideals which are not weakly thinnable: for instance, let $\mathcal{I}_f$ be the ideal defined by $f(2n)=1$ and $f(2n-1)=0$ for all $n \in \mathbf{N}$, so that 
$$
\mathcal{I}_f=\{I\subseteq \mathbf{N}: I \cap 2\mathbf{N} \in \mathrm{Fin}\}.
$$
Set $A:=\mathbf{N}\setminus \{1\}$ and $B:=2\mathbf{N}$. Then, $A$ has asymptotic density $1$, $B\notin \mathcal{I}_f$, and $A_B=2\mathbf{N}+1 \in \mathcal{I}_f$. Therefore $\mathcal{I}_f$ is not weakly thinnable.

\section{Main Results}\label{sec:mainresult}

Consider the natural bijection between the collection of all subsequences $(x_{n_k})$ of $(x_n)$ and real numbers $\omega \in (0,1]$ with non-terminating dyadic expansion 
$$
\sum_{i\ge 1}d_i(\omega)2^{-i},
$$
where $d_i(\omega)=1$ if $i=n_k$, for some integer $k$, and $d_i(\omega)=0$ otherwise, cf. \cite[Appendix A31]{MR1324786} and \cite{MR1260176}. Accordingly, for each $\omega \in (0,1]$, denote by $x \upharpoonright \omega$ the subsequence of $(x_n)$ obtained by omitting $x_i$ if and only if $d_i(\omega)=0$.

Moreover, let $\lambda: \mathscr{M}\to \mathbf{R}$ denote the Lebesgue measure, where $\mathscr{M}$ stands for the completion of the Borel $\sigma$-algebra on $(0,1]$. Our main result follows:
\begin{thm}\label{thm:main}
Let $\mathcal{I}$ be a thinnable ideal and $(x_n)$ be a sequence taking values in a first countable space $X$ where all closed sets are separable. Then 
$$
\lambda\left(\left\{\omega \in (0,1]: \Gamma_x(\mathcal{I})=\Gamma_{x\upharpoonright \omega}(\mathcal{I})\right\}\right)=1. 
$$
\end{thm}
\begin{proof}
Let $\Omega$ be the set of normal numbers, that is,
\begin{equation}\label{eq:normal}
\Omega:=\left\{\omega \in (0,1]: \lim_{n\to \infty}\frac{1}{n}\sum_{i=1}^n d_i(\omega)=\frac{1}{2}\right\}.
\end{equation}
It follows by Borel's normal number theorem \cite[Theorem 1.2]{MR1324786} that $\Omega \in \mathscr{M}$ and $\lambda(\Omega)=1$. Then, it is claimed that
\begin{equation}\label{eq:claim1}
\Gamma_{x\upharpoonright \omega}(\mathcal{I})\subseteq \Gamma_x(\mathcal{I})\,\,\,\text{ for all }\omega \in \Omega.
\end{equation}
To this aim, fix $\omega \in \Omega$ and denote by $(x_{n_k})$ the subsequence $x\upharpoonright \omega$. Let us suppose for the sake of contradiction that $\Gamma_{x\upharpoonright \omega}(\mathcal{I})\setminus \Gamma_x(\mathcal{I}) \neq \emptyset$ and fix a point $\ell$ therein. Then, the set of indexes $\{n_k: k\in \mathbf{N}\}$ has asymptotic density $\nicefrac{1}{2}$ and, for each neighborhood $U$ of $\ell$, it holds $\{k: x_{n_k} \in U\} \notin \mathcal{I}$. This implies 
$$
\{n: x_n \in U\} \supseteq \{n_k: x_{n_k} \in U\} \notin \mathcal{I}
$$
by the hypothesis that $\mathcal{I}$ is, in particular, weakly thinnable. Therefore $\{n: x_n \in U\} \notin \mathcal{I}$, which is a contradiction since $\ell$ would be also a $\mathcal{I}$-cluster point of $x$. This proves \eqref{eq:claim1}.

To complete the proof, it is sufficient to show that
\begin{equation}\label{eq:claim2}
\lambda\left(\left\{\omega \in (0,1]: \Gamma_x(\mathcal{I}) \subseteq \Gamma_{x\upharpoonright \omega}(\mathcal{I})\right\}\right)=1.
\end{equation}
This is clear if $\Gamma_x(\mathcal{I})$ is empty. Otherwise, note that $\Gamma_x(\mathcal{I})$ is closed by \cite[Lemma 3.1(iv)]{LMxyz}, hence there exists a non-empty countable dense subset $L$. 
Fix $\ell \in L$ and let $(U_m)$ be a decreasing local base of neighborhoods at $\ell$. Fix also $m \in \mathbf{N}$ and define $I:=\{n: x_n \in U_m\}$ which does not belong to $\mathcal{I}$; in particular, $I$ is infinite and we let $\{i_n: n\in \mathbf{N}\}$ be its enumeration. Again by Borel's normal number theorem, 
$$
\Theta(\ell,U_m):=\left\{\omega \in (0,1]: \lim_{n\to \infty}\frac{1}{n}\sum_{j=1}^n d_{i_j}(\omega)=\frac{1}{2}\right\}
$$
belongs to $\mathscr{M}$ and has Lebesgue measure $1$. Fix $\omega$ in the above set and denote by $(x_{n_k})$ the subsequence $x\upharpoonright \omega$. Hence, the set 
$
J:=\{n: i_n \in \{n_k: k \in \mathbf{N}\}\}
$ 
admits asymptotic density $\nicefrac{1}{2}$ and, by the thinnability of $\mathcal{I}$, we get $I_J \notin \mathcal{I}$. Lastly, note that 
$$
\{k: x_{n_k} \in U_m\} = \{k: n_k \in I\} \le \{n_k: n_k \in I\}=I_J.
$$ 
Therefore $\{k: x_{n_k} \in U_m\} \notin \mathcal{I}$. In addition, $\Theta(\ell):=\bigcap_{m\ge 1}\Theta(\ell,U_m)$ belongs to $\mathscr{M}$ and has Lebesgue measure $1$. This implies that 
$$
\lambda\left(\left\{\omega \in (0,1]: \ell \in \Gamma_{x\upharpoonright \omega}(\mathcal{I})\right\}\right)=1.
$$
(See also \cite[Theorem 1]{MR0493027} for the case $\mathcal{I}=\mathrm{Fin}$.) At this point, since $L$ is countable, we get  
$
\lambda\left(\left\{\omega \in (0,1]: L\subseteq \Gamma_{x\upharpoonright \omega}(\mathcal{I})\right\}\right)=1.
$ 
Claim \eqref{eq:claim2} follows by the fact that also $\Gamma_{x\upharpoonright \omega}(\mathcal{I})$ is closed by \cite[Lemma 3.1(iv)]{LMxyz}, so that each of these $\Gamma_{x\upharpoonright \omega}(\mathcal{I})$ contains the closure of $L$, i.e., $\Gamma_x(\mathcal{I})$.
\end{proof}

\begin{rmk}\label{rmk:separablemetric}
Separable metric spaces $X$ satisfy the hypotheses of Theorem \ref{thm:main}. Indeed, $X$ is first countable and every closed subset $F$ of $X$ is separable. To prove the latter, let $A$ be a countable dense subset of $X$ and note that
$$
\mathscr{F}:=\{B(a,r) \cap F: a \in A, 0<r \in \mathbf{Q}\}\setminus \{\emptyset\}
$$
is a base for $F$, where $B(a,r)$ is the open ball with center $a$ and radious $r$. Then, a set which picks one point for every set in $\mathscr{F}$ is a countable dense subset of $F$.
\end{rmk}

As a consequence of Proposition \ref{prop:ialphaideals},
Theorem \ref{thm:main}, and Remark \ref{rmk:separablemetric}, we obtain:
\begin{cor}\label{cor:alphadensity}
Let $x$ be a sequence taking values in a separable metric space. Then the set of statistical cluster points of $x$ is equal to the set of statistical cluster points of almost all its subsequences \textup{(}in the sense of Lebesgue measure\textup{)}.
\end{cor}

Similarly, setting $\mathcal{I}=\mathrm{Fin}$, we recover Buck's result \cite{MR0009997}:
\begin{cor}\label{cor:buck}
Let $x$ be a sequence taking values in a separable metric space. Then the set of ordinary limit points of $x$ is equal to the set of ordinary limit points of almost all its subsequences \textup{(}in the sense of Lebesgue measure\textup{)}.
\end{cor}

Lastly, we recall that a sequence $x=(x_n)$ taking values in topological space $X$ converges (with respect to an ideal $\mathcal{I}$) to $\ell \in X$, shortened as $x \to_{\mathcal{I}} \ell$, if
$$
\{n: x_n \notin U\} \in \mathcal{I}
$$
for all neighborhoods $U$ of $\ell$. 
In this regard, Miller \cite[Theorem 3]{MR1260176} proved that a real sequence $x$ converges statistically to $\ell$, i.e., $x \to_{\mathcal{I}_0}\ell$, if and only if almost all its sequences converge statistically to $\ell$. 

This is extended in the following result. Here, we say that an ideal $\mathcal{I}$ is \emph{invariant} if, for each $A\subseteq \mathbf{N}$ with positive asymptotic density, it holds $A_B \notin \mathcal{I}$ if and only if $B\notin \mathcal{I}$ (in particular, $\mathcal{I}$ is weakly thinnable). This condition is strictly related with the so-called ``property (G)'' defined in \cite{MR3568092}.
\begin{thm}\label{thm:oldmiller}
Let $\mathcal{I}$ be an invariant ideal and $x$ be a sequence taking values in a topological space. Then $x \to_\mathcal{I} \ell$ if and only if
$$
\lambda\left(\left\{\,\omega \in (0,1]: x\upharpoonright \omega \to_\mathcal{I} \ell \,\right\}\right)=1.
$$
\end{thm}
\begin{proof}
First, let us suppose that $x \to_\mathcal{I} \ell$ and let $U$ be a neighborhood of $\ell$. 
Let $\Omega$ be set of normal numbers defined in \eqref{eq:normal}, fix $\omega \in \Omega$, and denote by $(x_{n_k})$ the subsequence $x\upharpoonright \omega$. Then $I:=\{n:x_n \notin U\} \in \mathcal{I}$ and $A:=\{n_k:k \in \mathbf{N}\}$ has asymptotic density $\nicefrac{1}{2}$. Define $B:=\{k: x_{n_k}\notin U\}=\{k: n_k \in I\}$. Since $\mathcal{I}$ is, in particular, weakly thinnable and $A_B=\{n_k: x_{n_k} \notin U\} \in \mathcal{I}$, it follows that $B \in \mathcal{I}$, i.e., $x\upharpoonright \omega \to_{\mathcal{I}} \ell$.

Conversely, 
note that $\lambda(\Omega \cap (1-\Omega))=1$. Hence, there exists $\omega \in \Omega$ such that $x\upharpoonright \omega \to_\mathcal{I} \ell$ and $x\upharpoonright (1-\omega) \to_\mathcal{I} \ell$. It easily follows that $x \to_\mathcal{I} \ell$. Indeed, denoting by $(x_{n_k})$ and $(x_{m_r})$ the subsequences $x\upharpoonright \omega$ and $x\upharpoonright (1-\omega)$, respectively, we have that, for each neighborhood $U$ of $\ell$, it holds $\{k: x_{n_k} \notin U\} \in \mathcal{I}$ and $\{r: x_{m_r} \notin U\} \in \mathcal{I}$. 
Since $\{n_k: k \in \mathbf{N}\}$ and $\{m_r: r \in \mathbf{N}\}$ form a partition of $\mathbf{N}$, then
$$
\{n: x_n \notin U\}=\{n_k: x_{n_k} \notin U\} \cup \{m_r: x_{m_r} \notin U\}.
$$
The claim follows by the hypothesis that $\mathcal{I}$ is invariant.
\end{proof}

It is not possible to extend Theorem \ref{thm:oldmiller} on the class of all ideals: 
indeed, it has been shown in \cite[Example 2]{MR3568092} that there exists an ideal $\mathcal{I}$ and a real sequence $x$ such that $x \to_{\mathcal{I}} \ell$ and, on the other hand, $\lambda\left(\left\{\,\omega \in (0,1]: x\upharpoonright \omega \to_\mathcal{I} \ell \,\right\}\right)=0$.


\subsection*{Acknowledgments} 
The author is grateful to Piotr Miska (Jagiellonian University, PL) and Marek Balcerzak (\L{}\'{o}d\'{z} University of Technology, PL) for several useful comments.

\subsection*{Note added in proof} It turns out that the topological analogue of Theorem \ref{thm:main} is quite different, providing a non-analogue between measure and category. Indeed, it has been shown in \cite{Leo17c} that, if $x$ is a sequence in a separable metric space, then $\{\omega \in (0,1]: \Gamma_x(\mathcal{I}_0)=\Gamma_{x \upharpoonright \omega}(\mathcal{I}_0)\}$ is not a first Baire category set if and only if every ordinary limit point of $x$ is also a statistical cluster point of $x$, that is, $\Gamma_x(\mathrm{Fin})=\Gamma_x(\mathcal{I}_0)$.

\bibliographystyle{amsplain}
\bibliography{ideale}

%
%

\end{document}